\UseRawInputEncoding
\documentclass[10pt]{article}

\usepackage[dvips]{color}
\usepackage{epsfig}
\usepackage{float,amsthm,amssymb,amsfonts}
\usepackage{ amssymb,amsmath,graphicx, amsfonts, latexsym}
\usepackage{cleveref}

\begin{document}
\theoremstyle{plain}
\newtheorem{theorem}{Theorem}
\newtheorem{corollary}[theorem]{Corollary}
\newtheorem{lemma}[theorem]{Lemma}
\newtheorem{proposition}[theorem]{Proposition}
\newtheorem{remark}[theorem]{Remark}

\theoremstyle{definition}
\newtheorem{defn}{Definition}
\newtheorem{definition}[theorem]{Definition}
\newtheorem{example}[theorem]{Example}
\newtheorem{conjecture}[theorem]{Conjecture}

\title{ On certain Semigroup of Order-decreasing  Full Contraction Mappings of a Finite Chain}
\author{\bf  M. M. Zubairu\footnote{Corresponding Author. ~~Email: $mmzubairu.mth@buk.edu.ng$}~ A. Umar ~and J. A. Aliyu   \\
\it\small  Department of Mathematics, Bayero  University Kano, PM. Box 3011 Kano Nigeria\\
\it\small  \texttt{mmzubairu.mth@buk.edu.ng}\\[3mm]
\it\small Khalifa University, P. O. Box 127788, Sas al Nakhl, Abu Dhabi, UAE\\
\it\small \texttt{abdullahi.umar@ku.ac.ae}\\[3mm]
\it\small  Department of Mathematics, Bayero  University Kano, PM. Box 3011 Kano Nigeria\\
\it\small \texttt{aliyu.jafar@gmail.com}\\
}
\date{\today.}
\maketitle\

\begin{abstract}
 Let $\mathcal{CT}_n$ be the semigroup of full contraction mappings on $[n]=\{1,2,\ldots,n\}$, and let $\mathcal{OCT}_n$ and $\mathcal{ODCT}_n$ be the subsemigroups consisting of all order-preserving full contraction and subsemigroup of order-decreasing and order-preserving full contraction mappings, respectively. In this paper, we show that the semigroup $\mathcal{ODCT}_n$  is left adequate. We further study the rank properties and as well obtain the rank of the semigroup, $\mathcal{ODCT}_n$. Moreover, we obtain a characterization of natural partial order for the semigroup $\mathcal{OCT}_n$ and  its subsemigroup $\mathcal{ODCT}_n$, respectively.
 \end{abstract}
\emph{2010 Mathematics Subject Classification. 20M20.}

\section{Introduction}

 Denote $[n]$  to be a finite chain $\{1,2, \ldots ,n\}$ and  $\mathcal{P}_{n}$ (resp., $\mathcal{T}_n$) to be the semigroup of partial transformations on $[n]$ (resp., semigroup of full transformations on $[n]$). A map $\alpha\in \mathcal{T}_{n}$ is said to be \emph{order-preserving} (resp., \emph{order-reversing}) if  (for all $x,y \in [n]$) $x\leq y$ implies $x\alpha\leq y\alpha$ (resp., $x\alpha\geq y\alpha$); is \emph{order-increasing} (resp., \emph{order-decreasing}) if (for all $x\in [n]$) $x\leq x\alpha$ (resp., $x\alpha\leq x$);  an \emph{isometry} (i.e., \emph{ distance preserving}) if (for all $x,y \in [n]$) $\vert{x\alpha-y\alpha}\vert=\vert{x-y}\vert$ and   a \emph{contraction} if (for all $x,y \in [n]$) $\vert{x\alpha-y\alpha}\vert\leq \vert{x-y}\vert$. The collection of all contraction mappings on $[n]$  denoted by $\mathcal{CT}_{n}$ is known to be the \emph{semigroup of full contraction mappings}.
The study of this semigroup and some of its subsemigroups was first proposed in 2013 by Umar and Alkharousi \cite{af}. In this proposal, notations for the semigroup and some of its various subsemigroups were given. We will also maintain the same notations in this paper. Let

\begin{equation} \mathcal{OCT}_{n}=\{\alpha\in \mathcal{CT}_{n}:(\textnormal{for all }x,y\in [n])~x\leq y~ \Rightarrow~ x\alpha\leq y\alpha\} \end{equation}
be the semigroup of order-preserving full contractions,

\begin{equation} \mathcal{ORCT}_{n}=\mathcal{OCT}_{n}\cup\{\alpha\in \mathcal{CT}_{n}:( \textnormal{for all }x,y\in [n])~x\leq y~ \Rightarrow~  x\alpha\geq y\alpha ~\} \end{equation} be the semigroup of order-preserving or order-reversing full contractions and
\begin{equation}\label{od} \mathcal{ODCT}_{n}=\{\alpha\in \mathcal{OCT}_{n}:(\textnormal{for all }x \in [n])~ \Rightarrow~ x\alpha\leq x\} \end{equation}
be the semigroup of order-decreasing and order-preserving full contractions on $[n]$. Then it is clear that $\mathcal {ORCT}_n $ is a subsemigroup of $\mathcal{CT}_n$, while $\mathcal{ODCT}_n$ and $\mathcal{OCT}_n$ are subsemigroups of $\mathcal{ORCT}_n$. In 2013, Zhao and Yang \cite{py} characterized regular elements and all the Green's equivalences for the semigroup $\mathcal{OCP}_{n}$ (where $\mathcal{OCP}_{n}$ denote the semigroup of order-preserving partial contractions on $[n]$). Ali \emph{et al.,}\cite{mans1} extend the results of Zhao and Yang \cite{py} to a more general semigroup of partial contractions $\mathcal{CP}_n$. They obtained a  characterization for the regular elements and  all the Green's relations for the larger semigroup $\mathcal{CP}_{n}$. The complete characterization of Green's relations for the semigroup of full contraction $\mathcal{CT}_n$ were all obtained in \cite{mans1}. It is worth noting that the combinatorial results for the semigroups  $\mathcal{ORCT}_{n},$   $\mathcal{OCT}_{n}$ and $\mathcal{ODCT}_{n}$ were investigated by Adeshola and Umar \cite{au}. The cardinalities of the semigroups as well as the cardinality of the set of idempotents in each semigroup were obtained in \cite{au}. Furthermore,   Kemal \cite{kt} compute the ranks of the semigroups  $\mathcal{OCT}_{n}$ and $\mathcal{ORCT}_{n}$. Moreover, Leyla \cite{leyla} generalized  the work of Kemal \cite{kt} and obtained the rank of the two sided ideal of the semigroups  $\mathcal{OCT}_{n}$ and $\mathcal{ORCT}_{n}$. However, it appears that the rank and algebraic properties of the semigroup $\mathcal{ODCT}_{n}$ have not been investigated. In this paper, we intend to study the Green's relations, their starred analogue and rank properties. We also intend to characterize partial order relation on $\mathcal{OCT}_n$ and $\mathcal{ODCT}_n$.

In this section, we give a brief introduction,  basic definitions  and characterize the elements of $\mathcal{ODCT}_n$. In section 2, we characterize all the Green's equivalence and the regular elements in $\mathcal{ODCT}_n$. In section 3, we characterize the starred analogue of the Green's equivalences and show that $\mathcal{ODCT}_n$ is left abundant for all $n$ but not right abundant for $n\geq 3$. In section 4, we show that $\mathcal{ODCT}_n$ is left adequate and moreover, we investigate the rank of $\mathcal{ODCT}_n$. In section 5, we characterize partial order relation on the semigroups $\mathcal{OCT}_n$ and $\mathcal{ODCT}_n$.

 For a contraction $\alpha$ in $\mathcal{CT}_{n}$, we shall denote  Im$~\alpha$,  rank$~\alpha$ and id$_{A}$ to be the image  of $\alpha$, $|\textnormal{Im }\alpha|$ and identity on $A\subseteq[n]$, respectively. For $\alpha, ~\beta \in \mathcal{CT}_{n}$,  the composition of $\alpha$ and $\beta$ is defined as $x(\alpha \circ \beta) =((x)\alpha)\beta$ for all $x$ in $[n]$. We shall be using the notation $\alpha\beta$ to denote $\alpha \circ\beta$ in our subsequent discussions. An element $a$ in a semigroup $S$ is said to be an \emph{idempotent} if and only if $a^{2} = a$. It is well known that $\alpha \in \mathcal{T}_n$ is an idempotent if and only if Im$~\alpha = F(\alpha)$ (where $ F(\alpha) = \{x \in [n] : x\alpha = x \}$). If $S$ is a commutative semigroup and all  its elements are idempotents (i.e., $S = E(S)$), then $S$ is said to be a \emph{semilattice}. In this case, for all $\alpha,\beta \in S$, $\alpha^2 = \alpha$ and $\alpha\beta = \beta\alpha$. For basic concepts in semigroup theory, we refer the reader to  \cite{ph, howi}.

 Next, given any transformations $\alpha$ in $\mathcal{OCT}_n$, the domain of $\alpha$ is partitioned into \emph{blocks} by the relation ker~$\alpha=\{(x,y)\in[n]\times[n]:x\alpha=y\alpha\}$, so that as in \cite{py} $\alpha$  can be expressed  as \begin{equation}\label{1} \alpha=\left( \begin{array}{cccc}
            A_1 & A_2 & \ldots & A_p \\
            x+1 & x+2 & \ldots & x+p
          \end{array} \right) \ (1\leq p\leq n),
 \end{equation}\noindent where $A_i ~(1\leq i\leq p)$ are equivalence classes under the relation ker $\alpha$, i.e., $A_i=(x+i)\alpha^{-1} (1\leq i \leq p)$  (with  $A_i < A_j$ and $x+i<x+j$  if and only if $i<j$), moreover $(x+i)-(x+i-1)\leq \min A_{i}-\max A_{i-1}$, $i=2,\ldots,p-1$ (if $p>1$). We shall denote the partition of $[n]$ (by the relation ker $\alpha$) by \textbf{Ker} $\alpha=\{A_1,A_2, \ldots , A_p\}$ so that
$[n] = A_1\cup A_2 \cup \ldots \cup A_p$ where $(1\leq p\leq n)$.

  A subset $T_{\alpha}$ of $[n]$ is said to be a \emph{transversal} of the partition \textbf{Ker}$~\alpha$ if $\vert{T_{\alpha}} \vert = p$ and $\vert {A_i\cap T_\alpha} \vert = 1$  $(1\leq i \leq p)$. A transversal $T_\alpha$ of \textbf{Ker}$~\alpha$ is said to be \emph{convex} if for all $x,y \in T_\alpha$ with $x \leq y$ and if $x < z < y$ for $z \in [n]$ then $z \in T_\alpha.$  A transversal $T_\alpha$ is said to be \emph{admissible} if and only if the map $A_i\mapsto t_i$ $(t_i \in T_\alpha, i \in \{1,2, \ldots,p\})$ is a contraction. The partition \textbf{Ker}$~\alpha$ is said to be a \emph{convex partition} if it has a convex admissible transversal.

The elements in $\mathcal{ODCT}_n$ can be uniquely expressed as in the following lemma:
\begin{lemma}\label{nj} Every element   $\alpha\in\mathcal{ODCT}_n$ can be expressed as
$$ \alpha=\left( \begin{array}{cccc}
            A_1 & A_2 & \ldots & A_p \\
            1 & 2 & \ldots & p
          \end{array} \right)\, (1\leq p\leq n).$$
\end{lemma}
\begin{proof}
 Let $\alpha\in \mathcal{ODCT}_n$ be as expressed in equation \eqref{1}, i.e.,
$$ \alpha=\left( \begin{array}{cccc}
            A_1 & A_2 & \ldots & A_p \\
            x+1 & x+2 & \ldots & x+p
          \end{array} \right).$$ \noindent Since $\alpha $ is order-preserving then  $A_1<A_2< \ldots < A_p$ and   $x+1<x+2< \ldots < x+p$. Moreover, since  $\alpha$ is order-decreasing, then  $x+i\leq a$  for all  $ a\in A_i$ $ (1\leq i\leq n)$. In particular, $x+1\leq \min A_1$. Notice that  $1\in A_1$, which implies that $x+1= 1$, and so  $x=0$. Thus  $A_i\alpha=i$ for $1\leq i\leq p$, as required.
\end{proof}

Let
\begin{equation} \label{ab} \alpha=\left( \begin{array}{cccc}
            A_1 & A_2 & \ldots & A_p \\
            1 & 2 & \ldots & p
          \end{array} \right) \textnormal{ and } \beta=\left( \begin{array}{cccc}
            B_1 & B_2 & \ldots & B_p \\
            1 & 2 & \ldots & p
          \end{array} \right)\in \mathcal{ODCT}_{n},
\end{equation}\noindent then we have the following remark.
\begin{remark}\label{rem} For $\alpha, \beta\in \mathcal{ODCT}_{n}$:
\begin{itemize}
\item  If $|\textnormal{im }\alpha|=|\textnormal{im }\beta|$ then im $\alpha=$im $\beta$;
\item If $\textnormal{ker } \alpha= \textnormal{ker }\beta$ then $\alpha=\beta$.
\end{itemize}
\end{remark}

\section{Regularity and Green's relation}
For the definitions of the five Green's relations: $\mathcal{L}$, $\mathcal{R}$, $\mathcal{D}$, $\mathcal{H}$ and $\mathcal{J}$, we refer the reader to Howie \cite{howi} and Higgins \cite{hig}. It is well known that on a finite semigroup the relationss $\mathcal{D}$ and $\mathcal{J}$ are equal. The characterizations of Green's relations on various transformation semigroups have been investigated by many authors, for example see \cite{m1,grf1,grf2,grf3,py}. It is also well known that the semigroup of order-preserving and order-decreasing full transformation $\mathcal{C}_{n}$ is  $\mathcal{J}-$trivial, see for example \cite{hig1, hig, lu}. Thus we have the following result.

\begin{theorem} Let $\mathcal{ODCT}_n$ be as defined in equation \eqref{od}. Then $\mathcal{ODCT}_n$ is $\mathcal{J}-$trivial.
 \end{theorem}
 \begin{proof}
 Since the semigroup $\mathcal{ODCT}_n$ is a subsemigroup of  $\mathcal{C}_{n}$, then $\mathcal{ODCT}_n$ is $\mathcal{J}-$trivial.\end{proof}
  As a consequence, we have the following corollaries.

  \begin{corollary}  On the semigroup $\mathcal{ODCT}_n$, $\mathcal{L}=\mathcal{R}=\mathcal{D}=\mathcal{H}=\mathcal{J}$.
\end{corollary}
Now since $\mathcal{ODCT}_n$ is $\mathcal{R}$ trivial then we have the following.
\begin{corollary} An element $\alpha\in  \mathcal{ODCT}_n$ is regular if and only if $\alpha$ is an idempotent.
\end{corollary}

\section{Starred Green's relations}
Let $S$ be a semigroup.  An element $\alpha \in S$ is said to be \emph{regular} if and only if there exists $\gamma \in S$ such that $\alpha = \alpha\gamma\alpha$. A semigroup $S$ is said to be regular if all the elements of $S$ are regular. On a semigroup $S$, the relation $\mathcal{L}^*$ is defined by the rule that $(\alpha,\beta) \in\mathcal{L}^*$ if and only if $\alpha$ and $\beta$ are related by Green's $\mathcal{L}$ relation in some over semigroup of $S$. The relation $\mathcal{R}^*$ is defined dually and the relation $\mathcal{D}^*$ is defined as the join of the relations $\mathcal{L}^* \textnormal{ and } \mathcal{R}^*$, where the intersection of the relations $\mathcal{L}^* \textnormal{ and } \mathcal{R}^*$ is defined as $\mathcal{H}^*.$ A semigroup $S$ is said to be  \emph{right abundant}  (resp., \emph{left abundant}) if each $\mathcal{L}^* - class$ (resp., each $\mathcal{R}^* - class$) contains an idempotent, and is abundant if each $\mathcal{L}^* - \textnormal{class  and } \mathcal{R}^* - class$  of $S$ contains an idempotent. An abundant semigroup in which the set $E(S)$ is a subsemigroup is said to be \emph{quasi adequate} and if $E(S)$  is commutative then it is said to be \emph{adequate}, see \cite{mans1} for more details on adequate semigroups. Many classes of semigroups of transformations  were found to be regular and those that are not regular, their regular elements have been characterized, for example see \cite { m1,R1,R2,py}. If a semigroup is not regular, then their is  a need to investigate the class to which the semigroup belongs. To carry out such investigation one would naturally  characterize its starred Green's relations. In this section we investigate regularity, characterize the starred Green's relations and show that  $\mathcal{ODCT}_n$ is left  adequate. As in \cite{howi}, the relations $\mathcal{L}^* \textnormal{ and }\mathcal{R}^*$ have the following characterizations:
\begin{equation}\label{st1} \mathcal{L}^* =\{(\alpha,\beta) : (\textnormal{ for all } \mu,\lambda \in S^1)\alpha\mu = \alpha\lambda \textnormal{ iff } \beta\mu =\beta\lambda \}
\end{equation}\noindent and
\begin{equation}\label{st2} \mathcal{R}^* = \{(\alpha,\beta) : ( \textnormal{ for all } \mu,\lambda \in S^1)~~\mu\alpha = \lambda\alpha \textnormal{ iff } \mu\beta =\lambda\beta \}.
\end{equation}

We now give  characterizations of the Starred Green's relations on $\mathcal{ ODCT}_n$ in the theorem below. The proof of the theorem is a simplified version of the proof of  Theorem 1 in [\cite{am}].
\begin{theorem}\label{star}
Let $\alpha,\beta \in \mathcal{ODCT}_n$ be as expressed in equation \eqref{ab}. Then
\begin{itemize}
\item[(i)] $(\alpha,\beta)\in \mathcal{L}^\ast$ if and only if $\textnormal{Im }\alpha   =\textnormal{Im } \beta $;
\item[(ii)] $(\alpha,\beta)\in \mathcal{R}^\ast$ if and only if $\alpha=\beta$;
\item[(iii)]$\mathcal{H}^\ast=\mathcal{R}^\ast$;
\item[(iv)]$\mathcal{D}^\ast=\mathcal{L}^\ast$.
\end{itemize}
\end{theorem}
\begin{proof}
\begin{itemize}
\item[(i)] Let $\alpha,\beta \in \mathcal{ODCT}_n$ be as expressed in equation  \eqref{ab} and suppose  $(\alpha,\beta) \in \mathcal{L}^*$. Notice that $\textnormal{Im }\alpha = \{1,\ldots,p\}$. Now consider $\mu = \left( \begin{array}{cccc}
            1 & 2 & \ldots & \{i,\ldots,n\}\\
            1 & 2 & \ldots & i
          \end{array} \right) (1\leq p\leq i\leq n).$ \noindent Then clearly $\mu \in \mathcal{ODCT}_n$ and  $$\alpha \cdot \left( \begin{array}{cccc}
            1 & 2 & \ldots & \{i,\ldots,n\}\\
            1 & 2 & \ldots & i
          \end{array} \right) = \alpha \cdot id_{[n]}$$ \noindent if and only if  $$\beta \cdot \left( \begin{array}{cccc}
            1 & 2 & \ldots & \{i,\ldots,n\}\\
            1 & 2 & \ldots & i
          \end{array} \right) = \beta \cdot id_{[n]}~~(\textnormal{by equation }\eqref{st1})$$ which implies that   $\textnormal{Im}~\alpha\subseteq \textnormal{Im}~\beta$. We can  similar show that $\textnormal{Im}~\beta\subseteq \textnormal{Im}~\alpha$. Therefore $\textnormal{Im}~\alpha = \textnormal{Im}~\beta$.

Conversely, suppose  $\textnormal{Im}~\alpha= \textnormal{Im}~\beta$. Then by (\cite{howi}, Exercise 2.6.17) $\alpha \mathcal{L}^{\mathcal{P}_n}\beta$ and it follows from definition that $\alpha\mathcal{L}^*\beta$, the result follows.
          \item[(ii)] Let $\alpha,\beta \in \mathcal{ODCT}_n$ be as expressed in equation  \eqref{ab}. Suppose $(\alpha,\beta) \in \mathcal{R}^*$. Then $(x,y) \in \textnormal{ker }\alpha$ if and only if \begin{align*}x\alpha = y\alpha  & \Longleftrightarrow  \left( \begin{array}{cccc}
            [n] \\
            x
          \end{array} \right)\circ \alpha = \left( \begin{array}{cccc}
            [n] \\
            y
          \end{array} \right)\circ \alpha \\ &     \Longleftrightarrow  \left( \begin{array}{cccc}
            [n] \\
            x
          \end{array} \right)\circ \beta = \left( \begin{array}{cccc}
            [n] \\
            y
          \end{array} \right)\circ \beta~~ (\textnormal{by equation }\eqref{st2}) \\ & \Longleftrightarrow x\beta = y\beta \\ &\Longleftrightarrow (x,y) \in \textnormal{ker }\beta.\end{align*} \noindent Therefore  $\textnormal{ker }\alpha$ =  $\textnormal{ker }\beta$. Thus by Remark \ref{rem} (ii)  $\alpha = \beta$. The converse is clear. Thus  $(\alpha,\beta) \in \mathcal{R}^*$.
          \item[(iii)] The result follows from (i) and (ii).
\item[(iv)] Since $S$ is $\mathcal{R}^*$ trivial then $\mathcal{L}^*= \mathcal{D}^*$.
\end{itemize}
\end{proof}
On the semigroup $\mathcal{ODCT}_{n}$, the relations $\mathcal{D}^{*}$ and $\mathcal{J}^{*}$ are equal as we shall see below. However, we first note the following known result.

\begin{lemma}[\cite{f}, Lemma 1.7(3)]\label{j1}
Let $a,b$ be elements of a semigroup$S$. Then $b\in J^*(a)$ if and only if there are elements $a_0,a_1, \ldots,a_n \in S$, $x_1,x_2,\ldots,x_n,$ $y_1,y_2,\ldots,y_n, \in S^{1}$ such that $a = a_0$, $b = a_n$ and $(d_i, x_ia_{i-1}y_i)\in \mathcal{D}^*$ for $i = 1,2,\ldots,n$.
\end{lemma}

Now we prove an analogue of [\cite{ud}, Lemma 2.13.].
\begin{lemma}\label{ba}
Let $\alpha,\beta \in \mathcal{ODCT}_n$. If $\alpha\in J^*(\beta)$ then Im$~\alpha \subseteq \textnormal{Im}~\beta$.
\end{lemma}
\begin{proof}
 Let $\alpha\in \mathcal{J}^*(\beta)$, ($\alpha,\beta \in \mathcal{ODCT}_n$). Then by Lemma \ref{j1}, there are elements $\beta_0, \beta_1, \ldots,\beta_n \in \mathcal{ODCT}_n$, $\lambda_1,\lambda_2,\ldots,\lambda_n,$ $\mu_1,\mu_2,\ldots,\mu_n \in \mathcal{ODCT}_n^{1}$ such that $\beta = \beta_0,$ $\alpha = \beta_n$ and $(\beta_i, \lambda_i\beta_{i-1}\mu_i)\in \mathcal{D}^*$ for $i = 1,2,\ldots,n.$ Thus by Theorem \ref{star}(iv), $\textnormal{Im }\beta_i = \textnormal{Im }\lambda_i\beta_{i-1}\mu_i\subseteq \textnormal{Im }\beta_{i-1}$ for all $i = 1,2,\ldots,n$. This implies that $\textnormal{Im }\alpha \subseteq \textnormal{Im }\beta$.
\end{proof}

\begin{lemma}\label{ca}
On the semigroup $\mathcal{ODCT}_n$, $\mathcal{D}^* = \mathcal{J}^*$.
\end{lemma}

\begin{proof}
Let $\alpha,\beta \in \mathcal{ODCT}_n$. Notice that $\mathcal{D}^*\subseteq \mathcal{J}^*$. Thus we only need to show  that $\mathcal{J}^*\subseteq \mathcal{D}^*$. Now let $(\alpha,\beta)\in \mathcal{J}^*$, i.e., $\alpha \in \mathcal{J}^*(\beta)$ and $\beta\in \mathcal{J}^*(\alpha).$ Thus by Lemma \ref{ba}, $\textnormal{Im }\alpha \subseteq \textnormal{Im }\beta$ and $\textnormal{Im }\beta \subseteq \textnormal{Im }\alpha$, as such  $\textnormal{Im }\alpha = \textnormal{Im }\beta$.      Therefore by Theorem \ref{star} $(\alpha,\beta) \in \mathcal{D}^*$, as required.
\end{proof}

 Now we show in the  lemma below that $\mathcal{ODCT}_n$ is left abundant.
\begin{lemma}\label{la}
The semigroup $\mathcal{ODCT}_n$  is left abundant.
\end{lemma}
\begin{proof} Let $\alpha \in \mathcal{ODCT}_n$ be as expressed in equation \eqref{ab} and let $L^*_\alpha$ be an $\mathcal{L}^*$-class of $\alpha$ in $\mathcal{ODCT}_n$. Denote $$ \epsilon=\left( \begin{array}{cccc}
            1 & 2 & \ldots& \{p, p+1, \ldots,n\} \\
            1 & 2 & \ldots & p
          \end{array} \right) \in \mathcal{ODCT}_n, ~ (1\leq p\leq n).$$ \noindent It is clear that $\epsilon$ is an idempotent in $\mathcal{ODCT}_n$,  moreover $\textnormal{Im }\alpha=\textnormal{Im }\epsilon$ and so by  Theorem \ref{star}(i)  we see that $(\alpha,\epsilon) \in\mathcal{L}^*$ which means  $\epsilon \in{L}^*_\alpha.$ Since $L^*_\alpha$ is an arbitrary $\mathcal{L}^*$-class of $\alpha$ in $\mathcal{ODCT}_n$, then $\mathcal{ODCT}_n$ is left abundant, as required.
\end{proof}

\begin{remark}
In contrast with [\cite{dr}, Lemma 1.20], the semigroup $\mathcal{ODCT}_n$  is not right abundant for $n\geq 3$.

For a counterexample, consider $\alpha =\left( \begin{array}{cc}
            \{1,2\}& 3 \\
            1 & 2
          \end{array} \right) \in \mathcal{ODCT}_3$.  It is  clear that  $$R^*_\alpha =\left\{\left( \begin{array}{cc}
    \{1,2\}& 3 \\
            1 & 2
 \end{array} \right)\right\}$$ \noindent  has no idempotent.

 However, the semigroup $\mathcal{ODCT}_n$ is right abundant for $1\leq n\leq 2$ which is also in contrast with [\cite{dr}, Remark 1.21].
\end{remark}

\section{Rank of $\mathcal{ODCT}_n$}
Let $S$ be a semigroup and $A$ be any nonempty subset of $S$. The subsemigroup generated by $A$  is the smallest subsemigroup of $S$ containing $A$ and  is denoted by $\langle A \rangle.$ If there exists a finite subset $A$ of a semigroup $S$ with $\langle A\rangle  = S$, then $S$ is said to be a \emph{finitely generated semigroup}. The rank of a finitely generated semigroup $S$ is defined by
$$ \textnormal{ rank}(S)  = \min \{\vert{A} \vert :\langle A\rangle = S\}.$$ \noindent The ranks of many semigroups of transformations have been investigated over the years by many authors,  see for example \cite{gu,gur1,GH,ayik,u}. In particular, Kemal \cite{kt} obtained the ranks of the semigroups $\mathcal{OCT}_n$ and $\mathcal{ORCT}_n$, respectively. This study was extended to obtain the ranks  of the two sided ideals of $\mathcal{OCT}_n$ and $\mathcal{ORCT}_n$, respectively by Leyla \cite{leyla}. However, the rank of $\mathcal{ODCT}_n$ does not seem to have been investigated and in this section we investigate it.

Now let $\mathcal{ORCT}_{n}$ denote the semigroup of all order preserving or order reversing full contractions, and let $\textnormal{Reg}(\mathcal{ORCT}_{n})$ be the collection of regular elements of  $\mathcal{ORCT}_{n}$. Then, we first note the following result about idempotents in $\mathcal{ORCT}_{n}$ from \cite{am}.

\begin{lemma}[\cite{am}, Lemma 13]\label{mmm} Let $\epsilon$ be an idempotent element in $(\mathcal{ORCT}_{n})$.  Then $\epsilon$ can be expressed as $$\left(\begin{array}{cccccc}
                                                                            \{1,\ldots,a+1\} & a+2 & a+3 & \ldots & a+p-1 & \{a+p,\ldots,n\} \\
                                                                            a+1 & a+2 & a+3 & \ldots & a+p-1 & a+ p
                                                                          \end{array}
\right).$$
\end{lemma}

We now prove the following lemma which is crucial to the main result.
\begin{lemma}\label{2}
Every  $\epsilon \in E( \mathcal {ODCT}_n)$ can be expressed as
$$ \epsilon=\left( \begin{array}{cccc}
            1 & 2 & \ldots& \{p, p+1, \ldots,n\} \\
            1 & 2 & \ldots & p
          \end{array} \right)~~ (1\leq p\leq n).$$ \end{lemma}
\begin{proof} The proof follows from Lemma \ref{mmm} and Lemma \ref{nj}.
\end{proof}
We show in the next theorem that, the collection of all idempotents in $\mathcal{ODCT}_{n}$ i.e., $ E(\mathcal{ODCT}_{n})$ is a semilattice.

\begin{theorem}\label{sl}
$E (\mathcal {ODCT}_n)$ is a semilattice.
\end{theorem}
\begin{proof}
  Let $\epsilon,\eta\in E(\mathcal {OCDT}_n)$.  Then by Lemma \ref{2}, we may denote $\epsilon$ and $\eta$ by
$$ \epsilon =\left( \begin{array}{cccc}
            1 & 2 & \ldots& \{k, k+1, \ldots,n\} \\
            1 & 2 & \ldots&  k
          \end{array} \right) \textnormal{ and } \eta =\left( \begin{array}{cccc}
            1 & 2 & \ldots & \{p, p+1, \ldots,n\} \\
            1 & 2 & \ldots & p
          \end{array} \right) ~(1\leq k, p\leq n). $$
\noindent  Thus we have two cases to consider:

 If $k\leq p$. Then $$\epsilon\eta=\left( \begin{array}{cccc}
            1 & 2 & \ldots & \{k, k+1, \ldots,n\} \\
            1 & 2 & \ldots & k
          \end{array} \right)= \eta\epsilon = \epsilon \in E(\mathcal {ODCT}_n).$$ \noindent If $p< k$. Then $$\epsilon\eta=\left( \begin{array}{cccc}
            1 & 2 & \ldots & \{p, p+1, \ldots,n\} \\
            1 & 2 & \ldots & p
          \end{array} \right)=  \eta\epsilon = \eta \in E(\mathcal {ODCT}_n).$$ \noindent Thus $E(\mathcal {ODCT}_n)$ is a semilattice.
\end{proof}

Now by Theorem \ref{sl} and Lemma \ref{la}, we readily have the following result.
\begin{theorem} Let $\mathcal{ODCT}_n$ be as defined in equation \eqref{od}. Then $\mathcal{ODCT}_n$ is left adequate.
\end{theorem}

Next, we state the following well known result from \cite{jd} as a lemma below.
\begin{lemma}\label{dj} In a finite $\mathcal{J}$ trivial semigroup $S$, every minimal generating set is  (unique) minimum.

\end{lemma}
Let $G_p = \{\alpha \in \mathcal {ODCT}_n : \vert{\textnormal{Im}~\alpha}\vert = p\}$ and $K_{p} = \{\alpha \in \mathcal {ODCT}_n : |\textnormal{Im}~\alpha| \leq p\}$. It is worth noting that $K_p = G_1 \cup G_2 \cup \ldots \cup G_p$ $(1 \leq p \leq n)$. Now we have the following lemma.

\begin{lemma}\label{above1}
For  $1\leq p\leq n-2$,  $G_p \subseteq \langle G_{p+1}\rangle$.
\end{lemma}

\begin{proof} Let $\alpha\in G_{p}$, then by Lemma \ref{nj} we may let
$\label{alf} \alpha=\left( \begin{array}{cccc}
            A_1 & A_2 & \ldots & A_p\\
            1 & 2 & \ldots & p
          \end{array} \right)$, where $1\leq p\leq n-2$. Next now let $A_{p}'\cup A_{p}''=A_{p}$  with $A_{p}'\neq\emptyset$,  $A_{p}''\neq\emptyset$, $A_{p}'\cap A_{p}''=\emptyset$ and $A_{p}'< A_{p}''$. Now denote  $\delta$  and  $\rho$ as :
$$\label{alf} \delta=\left( \begin{array}{cccccc}
            A_1 & A_2 & \ldots &A_{p-1}& A'_p &A''_p\\
            1 & 2 & \ldots &p-1&p&p+1
          \end{array} \right)\ \textnormal{ and }\rho = \left( \begin{array}{cccccc}
            1 & 2 & \ldots&p-1 &\{p, p+1\}&\{p+2,\ldots,n\} \\
            1 & 2 & \ldots&p-1 & p & p+1
          \end{array} \right).
$$\noindent Notice that $\delta,\rho \in G_{p+1}$. It is easy to see that $\alpha = \delta\rho \in \langle G_{p+1}\rangle.$ Hence $G_{p} \subseteq \langle G_{p+1}\rangle$.
 \end{proof}

 As a consequence we have the following Corollary.
\begin{corollary}\label{C}
For $1\leq r\leq {n-1}$,  $G_r \subseteq \langle G_{n-1}\rangle$.
\end{corollary}
\begin{proof}
Suppose $1\leq r\leq {n-1},$ then by Lemma \ref{above1}, we see that $G_r \subseteq \langle G_{r+1}\rangle$ and similarly $G_{r+1} \subseteq \langle G_{r+2}\rangle$ which implies that $\langle G_{r}\rangle  \subseteq \langle G_{r+2}\rangle$. Therefore $G_r \subseteq \langle G_{r+1}\rangle  \subseteq \langle G_{r+2}\rangle$. If we continue in this fashion we see that
 $G_r \subseteq \langle G_{r+1}\rangle  \subseteq \langle G_{r+2}\rangle\subseteq \ldots \subseteq \langle G_{n-2}\rangle \subseteq \langle G_{n-1}\rangle$, as required.
\end{proof}

\begin{lemma} \label{order} In $\mathcal {ODCT}_n$,   $\vert  G_{n-1} \vert =n-1$.
\end{lemma}
\begin{proof}
Notice that if $\alpha \in G_{n-1}$ then $\alpha $  is of the  form
$\label{alf} \alpha=\left( \begin{array}{cccc}
            A_1 & A_2 & \ldots & A_{n-1} \\
            1 & 2 & \ldots & n-1
          \end{array} \right),$
  where $A_i<A_j$ if and only if $ i< j.$  It is now clear that order of  $ G_{n-1}$ is equal to the number of subsets of the set $[n]$ of the form $\{i, i+1\}$ ($1\leq i\leq n-1$) which is $n-1$.
\end{proof}

The following lemma gives us the rank of $\mathcal{ODCT}_n\setminus\{id_{n}\}$.

\begin{lemma}
In $\mathcal{ODCT}_n$, rank ($K_{n-1}) =n-1$.

\end{lemma}

\begin{proof}
To prove that the rank ($K_{n-1})= n-1$, it is enough to show that $G_{n-1}$ is a minimal generating set of
$K_{n-1}$, i. e,. $K_{n-1}= \langle G_{n-1} \rangle$ and $\langle G_{n-1} \backslash \{\tau\} \rangle \neq K_{n-1}$ for any $\tau \in G_{n-1}$. Notice that by Corollary \ref{C}, $G_1\subseteq \langle G_{n-1} \rangle$,$G_2 \subseteq \langle G_{n-1} \rangle$,\ldots,$G_{n-1} \subseteq \langle G_{n-1} \rangle$. Thus it easily follows that $G_1\cup G_2 \cup \ldots \cup G_{n-1} \subseteq \langle G_{n-1} \rangle$, i.e., $K_{n-1}\subseteq G_{n-1}$, i.e., $\langle G_{n-1} \rangle = K_{n-1}$.

Notice that
$$ \label{alf} G_{n-1}= \left\{ \left( \begin{array}{cccc}
            \{1,2\} & 3 &  \ldots & n \\
            1 & 2 &   \ldots & n-1
          \end{array} \right),\left( \begin{array}{ccccc}
            1, & \{2,3\} & 4& \ldots & n \\
            1 & 2 & 3 &  \ldots & n-1
          \end{array} \right),\ldots,\left( \begin{array}{cccccc}
            1 & 2 & 3 & \ldots & n-2 & \{n-1,n\} \\
            1 & 2 & 3  &\ldots& n-2 & n-1
          \end{array} \right)\right\}.$$ Take $\tau_{i}=\left( \begin{array}{cccccccc}
            1 & 2 & \ldots & \{i,i+1\}& \ldots& n-2 &n-1 &n \\
            1 & 2 & \ldots& i& \ldots& n-3 & n-2&n-1
          \end{array} \right)\in G_{n-1}$ for $i=1,\ldots,n-1$. \noindent Then one can easily verify that for any $\alpha,\beta \in G_{n-1}\backslash\{\tau_{i}\}$, $h(\alpha\beta)<{n-1}$, $h(\tau_{i}\alpha)<n-1$, $h(\alpha\tau_{i})<n-1$  and moreover, $\alpha\tau_{n-1}=\alpha$ for all $\alpha\in G_{n-1}$. Thus $G_{n-1}$ is a minimal generating set for $K_{n-1}$. Thus since $S$ is a finite $\mathcal{J}$ trivial semigroup then by Lemma \ref{dj}, $G_{n-1}$ is the (unique) minimum generating set for $K_{n-1}$.
\end{proof}
Finally the rank of $\mathcal{ODCT}_n$ is given in the  theorem below.
\begin{theorem} Let $\mathcal{ODCT}_n$ be as defined in equation \eqref{od}.
Then rank ($\mathcal {ODCT}_n) = n$.
\end{theorem}

\begin{proof}
Notice that $K_{n-1} =\mathcal {ODCT}_n \backslash \{id_{[n]}\}$. Therefore  the rank ($\mathcal {ODCT}_{n}) = \textnormal{ rank } (K_{n-1}) +1=n$, as required.
\end{proof}

\section{Natural Partial Order on the semigroup  $\mathcal{OCT}_{n}$}

Let $S$ be a semigroup and $\alpha,\beta \in S$. Define a relation $\leq$ on a semigroup $S$ by: $\alpha \leq \beta$ ($\alpha, \beta\in S$) if and only if there exist $\lambda,\mu \in S^{1}$ such that $\alpha = \lambda\beta = \beta\mu$ and $\alpha = \alpha\mu$. This relation is known to be the \emph{natural partial order} on a semigroup $S$. If $S$ is regular then $\alpha \leq \beta$ if and only if $\alpha = \epsilon\beta = \beta\eta$ for some $\epsilon,\eta \in E(S)$, and if $S$ is a semilattice of idempotents then $\epsilon \leq \eta$ if and only if $\epsilon = \epsilon\eta = \eta\epsilon$ for some $\epsilon,\eta \in E(S)$.  Partial order relations on various semigroups of the partial transformations   have been investigated by many authors,  see for example  \cite{po1,po5,po6,po7}. It is worth noting that the semigroup $\mathcal{OCT}_n$ is not regular (see \cite{dr}). In this section, we characterize the partial order relation defined above on the semigroups $\mathcal{OCT}_n$ and $\mathcal{ODCT}_n$, respectively.

 Let $\alpha,\beta$ be as expressed in equation \eqref{1}.  Consider \begin{equation} \alpha\beta^{-1} = \{(x,y)\in[n]\times [n]: x\beta = y\alpha \},
\end{equation}
\begin{equation} \alpha\alpha^{-1} = \{(x,y)\in[n]\times [n]: x\alpha = y\alpha \}.
\end{equation}

 Before we begin our investigation we first acknowledge the following known result from \cite{po6}.

\begin{lemma}[\cite{po6}, Theorem 2]\label{rpo1} Let $\alpha,\beta \in \mathcal{P}_n$. Then   $\alpha \leq \beta$ if and only if $\textnormal{dom }\alpha \subseteq \textnormal{dom}~\beta$,  $\textnormal{im}~\alpha \subseteq \textnormal{im}~\beta$, $\alpha\beta^{-1} \subseteq \alpha\alpha^{-1}$ and $\beta\beta^{-1} \cap (\textnormal{dom}~\beta \times \textnormal{dom}~\alpha) \subseteq \alpha\alpha^{-1}$.
\end{lemma}
Now we have the following lemma.
\begin{lemma}\label{rpo2} Let $\alpha,\beta \in \mathcal{OCT}_n$. If $\alpha\beta^{-1} \subseteq \alpha\alpha^{-1}$. Then $y\beta^{-1} \subseteq y\alpha^{-1}$ for all $y \in \textnormal{im}~\alpha$.
\end{lemma}
\begin{proof} Suppose $\alpha\beta^{-1} \subseteq \alpha\alpha^{-1}$, $y \in \textnormal{im}~\alpha$ and let $x \in y\beta^{-1}$. Then there exists $b \in [n]$ such that $b\alpha=y$ and $y = x\beta$, i.e., $b\alpha=x\beta$, which imply  $(x,b) \in \alpha\beta^{-1}$. Thus, by our assumption, we have $(x,b) \in \alpha\alpha^{-1}$, so that $x\alpha = b\alpha=y$ which implies that . Hence $x \in  y\alpha^{-1}$, as required.
\end{proof}
Let $\alpha$ in $\mathcal{OCT}_n$ be of rank 1. Then we have the following.

\begin{lemma}\label{po1}
Let $\alpha, \ \beta \in \mathcal{OCT}_n$ be such that $\alpha = \left( \begin{array}{cccc}
            [n] \\
            x
          \end{array} \right)$. Then $\alpha \leq \beta$ if and only if $x \in \textnormal{im}~\beta$.
\end{lemma}
\begin{proof} Let $\alpha,\beta \in \mathcal{OCT}_n$, where  $\alpha = \left( \begin{array}{c}
            [n] \\
            x
          \end{array} \right)$. Suppose $\alpha \leq \beta$, i.e., there exist $\lambda,\mu \in \mathcal{OCT}_n$ such that $\alpha = \lambda\beta = \beta\mu$ and $\alpha\mu = \alpha$. Notice that $\alpha = \lambda\beta$. It is easy  to see that $\textnormal{im}~\alpha \subseteq \textnormal{im}~\beta$. Therefore $x \in \textnormal{im}~\beta$.

Conversely, suppose $x \in \textnormal{im}~\beta$. Define $\lambda = \left( \begin{array}{c}
            [n] \\
            y
          \end{array} \right)$ where $y \in x\beta^{-1}$ and $\mu = \alpha$. Thus, it follows that $\alpha = \lambda\beta = \beta\mu$ and $\alpha\mu = \alpha$, and therefore  $\alpha \leq \beta$.
\end{proof}
Now if $\alpha$ is of rank 2, we have the following.
\begin{lemma}\label{po2}
Let $\alpha,\beta \in \mathcal{OCT}_n$ be such that $\alpha = \left( \begin{array}{cc}
            A_1 & A_2 \\
            x+1 & x+2
          \end{array} \right)$. Then $\alpha \leq \beta$ if and only if  $\textnormal{im}~\alpha \subseteq \textnormal{im}~\beta$ and $(\max A_1)\beta = x+1$ and $(\min A_2)\beta = x+2$.
\end{lemma}

\begin{proof}
Let $\alpha,\beta \in \mathcal{OCT}_n$, where $\alpha = \left( \begin{array}{cc}
            A_1 & A_2 \\
            x+1 & x+2
          \end{array} \right)$. Suppose $\alpha \leq \beta$ i.e., there exist $\lambda,\mu \in \mathcal{OCT}_n$ such that $\alpha = \lambda\beta = \beta\mu$ and $\alpha\mu = \alpha$. Since $\alpha = \lambda\beta$ then obviously $\textnormal{im}~\alpha \subseteq \textnormal{im}~\beta$, So $\textnormal{im}~\alpha \subseteq \textnormal{im}~\beta$. Suppose by way of contradiction that $(\max A_1)\beta = k \neq x+1 $ for some $k \in \textnormal{im}~\beta$. Since $\beta$ is order-preserving then we have that \begin{equation}\label{pe1} x+1 < k < x+2.
\end{equation}
Notice that $\alpha = \lambda\beta = \beta\mu$ and $\alpha\mu = \alpha$, for some $\lambda,\mu \in \mathcal{OCT}_n$. Now $\alpha = \beta\mu$ ensure that $k\mu = x+1$ and $(x+i)\mu = x+i$ for all $i \in \{1,2\}$. Thus $\vert{(x+2) - (x+1)}\vert = |(x+2)\mu- k\mu| \leq |(x+2) - k|$ which contradicts equation \eqref{pe1}. Hence $ (\max A_1)\beta = x+1$. In a similar way, one can easily show that $ (\min A_2)\beta = x+2$.

Conversely, suppose that condition (i) and (ii) holds. Define  $\lambda$ as :  $$ y\lambda = \left\{
                                                                                            \begin{array}{ll}
                                                                                              \max A_1, & \textnormal{ if } \hbox{$y \in A_1$;} \\
                                                                                              \min A_2, & \textnormal{ if } \hbox{$y \in A_2$.}
                                                                                            \end{array}
                                                                                          \right.
$$
and $$\mu = \left( \begin{array}{cc}
            \{1, \ldots, x+1\} & \{x+2, \ldots,n \} \\
            x+1 & x+2
          \end{array} \right)\in \mathcal{OCT}_n.$$ \noindent We now show that $\lambda \in \mathcal{OCT}_n$. Notice that if $y_1,y_2 \in A_1$ then $$|y_1\lambda-y_2\lambda| = |\max A_1-\max A_1| = |y_1-y_1| \leq |y_1-y_2|.$$\noindent
Now if $y_1,y_2 \in A_2$. Then $$|y_1\lambda-y_2\lambda| = |\min A_2-\min A_2| = |y_2-y_2| \leq |y_1-y_2|.$$

 Finally, if $y_1 \in A_1$ and $y_2 \in A_2$, then $$|y_1\lambda-y_2\lambda| = |\max A_1-\min A_2| = |a-b| \textnormal{ for all } a \in A_1,~b \in A_2.$$ \noindent In particular, $|y_1\lambda-y_2\lambda|\leq |y_1-y_2|$. Hence $\lambda$ is a contraction and since $A_1<A_2$ then $\lambda\in \mathcal{OCT}_n$.  It is now easy to see that $\alpha = \lambda\beta = \beta\mu$ and $\alpha\mu = \alpha$. Thus $\alpha \leq \beta$, as required.
\end{proof}
 Next let $\alpha\in \mathcal{OCT}_n$ be of rank $p$. Then we have the following.
\begin{theorem}\label{po3}
Let $\alpha,\beta \in \mathcal{OCT}_n$ be such that $\alpha = \left( \begin{array}{cccccc}
            A_1 & A_2 &\ldots &A_{p-1} & A_p \\
            x+1 & x+2 &\ldots& x+p-1& x+p
          \end{array} \right)$, $(3\leq p\leq n)$. Then $\alpha \leq \beta$ if and only if
\begin{itemize}
\item[(i)] $\textnormal{im}~\alpha \subseteq \textnormal{im}~\beta$;
\item[(ii)] $(x+i)\beta^{-1} = A_i$ for all $i \in \{2,3, \ldots, p-1 \}$ and
\item[(iii)] $(\max A_1)\beta = x+1 \textnormal{ and } (\min A_p)\beta = x+p$.
\end{itemize}
\end{theorem}

\begin{proof}
\begin{itemize}
\item[(i)] Let $\alpha,\beta \in \mathcal{OCT}_n$, where $\alpha = \left( \begin{array}{cccccc}
            A_1 & A_2& \ldots& A_{p-1} & A_p \\
            x+1 & x+2& \ldots &x+p-1& x+p
          \end{array} \right)$. Suppose $\alpha \leq \beta$. Then (i) follows obviously from the proof of Theorem \ref{po2}(i).
\item[(ii)] Since $\mathcal{OCT}_n \leq \mathcal{P}_n$ then by Lemma \ref{rpo1}, we have that $\alpha\beta^{-1}\subseteq \alpha\alpha^{-1}$. Moreover, by Lemma \ref{rpo2} we have that $(x+i)\beta^{-1} \subseteq A_i$ for all $i \in \{1,2, \ldots, p \}$. Suppose by way of contradiction that $(x+i)\beta^{-1} \neq A_i$  i.e., $(x+i)\beta^{-1} \subset A_i$ for some $i \in \{2,3,\ldots,p-1 \}$. Let $a \in A_i \backslash (x+i)\beta^{-1}$. Then there exists $b \in \textnormal{m}~\beta\backslash \{x+i\}$ such that $a \in b\beta^{-1}$. Since $\alpha = \beta\mu$ then $b \in \textnormal{dom}~\mu$ and $b\mu = x+i$ for some $i \in \{2,3, \ldots, p-1\}$. Since $\alpha = \alpha\mu$ then $x+i = (x+i)\mu$ for all $i \in \{2,3,\ldots,p-1\}$. There are two cases to consider about $b$.
Case 1 :  If $x+i < b \leq x+i+1$. Then
\begin{equation}\label{r31}\vert{(x+i+1)- b}\vert<\vert{(x+i+1)- (x+i)}\vert.
\end{equation}
\noindent Now $\vert{(x+i+1) -(x+i)}\vert = \vert{(x+i+1)\mu- b\mu}\vert \leq \vert{(x+i+1)- b}\vert$ which contradicts equation \eqref{r31}.

Case 2 :  If $x+i-1 \leq b < x+i$. Then
\begin{equation}\label{r32}\vert{b - (x+i-1)}\vert<\vert{(x+i)- (x+i-1)}\vert.
\end{equation}
\noindent Now $\vert{(x+i)- (x+i-1)}\vert = \vert{b\mu - (x+i-1)\mu}\vert \leq \vert{b - (x+i-1)}\vert$ which contradicts equation \eqref{r32}. So we have  $(x+i)\beta^{-1} = A_i$ for some $i \in \{2,3, \ldots, p-1 \}$
\item[(iii)]   Suppose by way of contradiction that $ (\max A_1)\beta = k \neq x+1 $ for some $k \in \textnormal{im}~\beta$. Since $\beta$ is order-preserving, we have that \begin{equation}\label{pe2} x+1 < k < x+2.
\end{equation}
\noindent Notice that $\alpha = \lambda\beta = \beta\mu$ and $\alpha\mu = \alpha$, for some $\lambda,\mu \in \mathcal{OCT}_n$. Now $\alpha = \beta\mu$ ensure that $k\mu = x+1$ and $(x+i)\mu = x+i$ for all $i \in \{1,2,\ldots,p\}$. Thus $\vert{(x+2) - (x+1)}\vert = |(x+2)\mu- k\mu| \leq |(x+2) - k|$ which contradicts equation  \eqref{pe2}. Hence $ (\max A_1)\beta = x+1$. Similarly, one can show that $ (\min A_p)\beta = x+p$.

Conversely, suppose that conditions (i) - (iii) holds. Define $\lambda$  as:  $$y\lambda = \left\{
                                                                                 \begin{array}{ll}
                                                                                   \max A_1, & \textnormal{ if } \hbox{$y \in A_1$;} \\
                                                                                   y, & \textnormal{ if } \hbox{$y \in A_2\cup \ldots\cup A_{p-1}$;} \\
                                                                                   \min A_p, & \textnormal{ if } \hbox{$y \in A_p$.}
                                                                                 \end{array}
                                                                               \right.
$$ and $\mu = \left( \begin{array}{cccc}
            \{1, \ldots, x+1\} & x+2 \ldots x+p-1 & \{x+p, \ldots,n \} \\
            x+1 & x+2 \ldots x+p-1 & x+p
          \end{array} \right)\in \mathcal{OCT}_n$. Then $\lambda \in \mathcal{OCT}_n$. To see this, notice that if $y_1,y_2 \in A_1$ then $$\vert{y_1\lambda-y_2\lambda}\vert = \vert{\max A_1-\max A_1}\vert = \vert{y_1-y_1}\vert \leq \vert{y_1-y_2}\vert.$$

Now if $y_1,y_2 \in A_p$. Then $$\vert {y_1\lambda-y_2\lambda}\vert = \vert {\min A_p-\min A_p}\vert = \vert {y_p-y_p}\vert \leq \vert {y_1-y_2}\vert.$$\noindent

Also, if $y_1,y_2 \in A_i$ for all $i\in \{2,\ldots,p-1\}$. Then $\vert {y_1\lambda-y_2\lambda}\vert = \vert {y_1-y_2}\vert\leq \vert {y_1-y_2}\vert$.

 Finally, if $y_1 \in A_1$ and $y_2 \in A_p$, then $\vert{y_1\lambda-y_2\lambda}\vert = \vert {\max A_1-\min A_p}\vert = \vert {a-b}\vert \textnormal{ for all } a \in A_1,~b \in A_p$. In particular $$\vert {y_1\lambda-y_2\lambda}\vert\leq \vert {y_1-y_2}\vert.$$ Hence $\lambda$ is a contraction and since $A_i<A_j$ if and only if $i<j$ then $\lambda\in \mathcal{OCT}_n$. It is now easy to see that $\alpha = \lambda\beta = \beta\mu$ and $\alpha\mu = \alpha$. Thus $\alpha \leq \beta$.

\end{itemize}
\end{proof}
Now we illustrate the constructions behind the proof of Theorem \ref{po3} with an example.
\begin{example} For $n = 10$.

$$\alpha = \left( \begin{array}{ccccc}
            \{1,2,3\} & \{4,5\} & 6 & \{7,8\} &\{9,10\} \\
            4 & 5 & 6 & 7 & 8
          \end{array} \right)$$
and
$$\beta_1 = \left( \begin{array}{ccccccc}
            \{1,2\} & 3 & \{4,5\} & 6 & \{7,8\} & 9 & 10 \\
            3 & 4 & 5 & 6 & 7 & 8 & 9
          \end{array} \right).$$
Then it is easy to check that $\alpha$ and $\beta$ satisfies the conditions (i) - (iii) of Theorem \ref{po3}. Now
denote $$\lambda = \left( \begin{array}{cccccccc}
            \{1,2,3\} & 4 & 5 & 6 & 7& 8&\{9,10\} \\
            3 & 4 & 5 & 6 & 7 & 8 & 9
          \end{array} \right)$$ \noindent and

$$ \mu = \left( \begin{array}{ccccc}
            \{1,2,3,4\} & 5 & 6 & 7 & \{8,9,10\} \\
            4 & 5 & 6 & 7 & 8
          \end{array} \right).$$ \noindent One can easily check  that $\alpha = \lambda\beta = \beta\mu$ and $\alpha\mu = \alpha$. Thus $\alpha\leq\beta_1$.

 It worth noting that if

$$\beta_2 = \left( \begin{array}{ccccccc}
            \{1,2\} & 3 & 4 & 5 & \{6,7,8\} & 9 & 10 \\
            3 & 4 & 5 & 6 & 7 & 8 & 9
          \end{array} \right).$$ \noindent Then clearly $\alpha\nleq\beta_2$. To see this, suppose that there exists $\mu \in \mathcal{OCT}_n$ such that $\alpha = \beta_2\mu$. Then notice that $$ 6 = (6)\alpha = (6)\beta_2\mu = (7)\mu \textnormal{  and  } 8 = (9)\alpha = (9)\beta_2\mu = (8)\mu.$$ \noindent Thus $\vert{8 -  6}\vert = \vert{(8)\mu - (7)\mu}\vert \nleq \vert{8 -7}\vert$, which is a contradiction. Therefore $\alpha\nleq \beta_2$  because $\beta_2$ does not satisfy condition (ii) in Theorem \ref{po3}.

\end{example}

We now deduce a characterization of partial order relation on the subsemigroup $\mathcal{ODCT}_n$ from the  results obtained for the semigroup  $\mathcal{OCT}_n$. But before then, it is worth noting that the element $ \left( \begin{array}{cccc}
            [n] \\
            x
          \end{array} \right) $ is the only element of rank 1 in $\in \mathcal{ODCT}_n$. Thus we have the following result.

\begin{corollary}\label{po4}
Let $\alpha = \left( \begin{array}{cccc}
            [n] \\
            1
          \end{array} \right) \in \mathcal{ODCT}_n$. Then $\alpha \leq \beta$ for all $\beta \in \mathcal{ODCT}_n$.
\end{corollary}
\begin{proof}
Let $\alpha = \left( \begin{array}{cccc}
            [n] \\
            1
          \end{array} \right)$ and $\beta = \left( \begin{array}{cccc}
            B_1 &  \ldots  & B_p \\
            1 &  \ldots & p
          \end{array} \right)$. Denote $\lambda = \mu = \left( \begin{array}{cccc}
            [n] \\
            1
          \end{array} \right)$. Therefore  $\alpha = \lambda\beta = \beta\mu$ and $\alpha\mu = \alpha$. Thus $\alpha \leq \beta$.
\end{proof}

We next deduce a characterization of partial order relation on $\mathcal{ODCT}_n$ (if $\alpha$ is of rank 2) from Theorem \ref{po2}.

\begin{corollary}\label{po5}
Let $\alpha,\beta \in \mathcal{ODCT}_n$ be such that $\alpha = \left( \begin{array}{cccc}
            A_1 & A_2 \\
            1 & 2
          \end{array} \right)$. Then $\alpha \leq \beta$ if and only if $(\max A_1)\beta = 1 \textnormal{ and } (\min A_2)\beta = 2$.
\end{corollary}
\begin{proof}
Let $\alpha,\beta \in \mathcal{ODCT}_n$, where $\alpha = \left( \begin{array}{cccc}
            A_1 & A_2 \\
            1 & 2
          \end{array} \right)$. Suppose $\alpha \leq \beta$. Since $ \mathcal{ODCT}_n\subset \mathcal{OCT}_n$ then $\alpha,\beta \in \mathcal{OCT}_n$.  Thus by Theorem \ref{po2} we have $\textnormal{im}~\alpha \subseteq \textnormal{im}~\beta$ and $(\max A_1)\beta = 1$ and $(\min A_2)\beta = 2$. Notice  that $\textnormal{im}~\alpha=\{1,2\}\subseteq \textnormal{im}~\beta$ ($p\geq 2$) always holds by Lemma \ref{nj}.

Conversely, $(\max A_1)\beta = 1 \textnormal{ and } (\min A_2)\beta = 2$. Then clearly, $\textnormal{im}~\alpha=\{1,2\}\subset \textnormal{im}~\beta=\{1,2,\ldots,p\}$ ($p\geq 2$) and $(\max A_1)\beta = 1+0 \textnormal{ and } (\min A_2)\beta = 2+0$, and therefore by Theorem \ref{po2} we have $\alpha \leq \beta$, as required.
\end{proof}

Now we have the following corollary and its proof follows from Theorem \ref{ab} which is similar to the proof of Corollary \ref{po5}, thus we omit it.

\begin{corollary}\label{po6}
Let $\alpha,\beta \in \mathcal{ODCT}_n$ be such that $\alpha = \left( \begin{array}{cccc}
            A_1 &  \ldots  & A_p \\
            1 &  \ldots & p
          \end{array} \right)$, $(1\leq p\leq n)$. Then $\alpha \leq \beta$ if and only if
\begin{itemize}
\item[(i)] $(\max A_1)\beta = 1 \textnormal{ and } (\min A_p)\beta = p$; and
\item[(ii)] $(i)\beta^{-1} = A_i$ for some $i \in \{2,3, \ldots, p-1 \}$.
\end{itemize}
\end{corollary}

\end{document}